\documentclass[12pt]{article}
\usepackage{geometry}    
\usepackage{graphicx}
\usepackage{amssymb}
\usepackage{epstopdf}
\usepackage{amsthm}

\newtheorem{theorem}{Theorem}
\newtheorem{lemma}{Lemma}

\newtheorem{claim}{Claim}

\title{On heterochromatic out-directed spanning \\\ trees in tournaments\footnote{Research partially supported by Conacyt, M\'exico and by PAPIIT-M\'exico  project  IN101912.}}

\author{Juan Jos\'e Montellano-Ballesteros  \\ {\it Instituto de Matem\'aticas, }\\\textit{ Universidad Nacional Aut\'onoma de M\'exico} \\ {\tt juancho@math.unam.mx}\\ \\ Eduardo Rivera-Campo  \\ {\it Departamento de Matem\'aticas }\\\ \textit{Universidad Aut\'onoma Metropolitana-Iztapalapa} \\  {\tt erc@xanum.uam.mx} }

\date{}

\begin{document}

\maketitle

\begin{abstract}
Given a tournament $T$, let $h(T)$ be the smallest integer $k$ such that  every arc-coloring of $T$ with $k$ or more colors produces at least  one out-directed spanning  tree of $T$ with no pair of arcs with the same color. In this paper we give the exact value of $h(T)$.\\\


\end{abstract}

\section{Introduction}

Given a graph $G$ and an edge-coloring of $G$, a subgraph $H$ of $G$  is said to be {\it heterochromatic}  if no pair of edges of $H$ have the same  color. Problems concerning the existence of heterochromatic subgraphs with a specific property in edge-colorings of a host graph are known as  {\it anti-Ramsey problems} (see, for instance,  \cite{AHK,ES,HY,KL,MN,SS}). Typically, the host graph $G$ is a complete graph or some graph  with a particular structure, and the property which defines the set of heterochromatic  subgraphs in consideration is that they are  {\it isomorphic to a given graph $H$} or that they are subgraphs  of $G$ with a general property like, for example, being edge-cuts or spanning trees of  $G$  (see \cite{BV,CHH,JW,Mo,M}).  
 
 A  {\it tournament}  is a digraph $D= (V(D), A(D))$ such that for every pair $\{x, y\} \subseteq V(D)$, either $xy \in A(D)$ or $yx\in A(D)$ but not both. A spanning tree $S$ of a tournament $T$ is an  {\it out-directed spanning tree} of $T$ if there is a root vertex $r$ of $S$ such that for each vertex $u \in V(S)$, the unique $r-u$ path in $S$ is directed from $r$ to $u$. 

In this paper,  the host graphs are tournaments, and the property that defines the set of heterochromatic subgraphs in consideration is that of being an out-directed spanning tree of the corresponding tournament.

Let $T= (V(T), A(T))$ be a tournament.  An {\it arc-coloring of $T$} is a function $\Gamma: A(T)\rightarrow C$, where $C$ is a set of  ``colors"; if $|\Gamma[A(T)]| = k$ we say that $\Gamma$ is a $k$-arc-coloring of $T$. 
A subdigraph $H$ of $T$  is said to be {\it heterochromatic}  if no pair of arcs of $H$ have the same  color. We define $h(T)$ as the smallest integer $k$ such that every $k$-arc-coloring of  $T$ produces at least one heterochromatic-out  directed spanning tree of $T$. Our main result is the following theorem:

 \begin{theorem}  \label{teorema}  Let $T$ be a tournament of order $n\geq 3$. Then $h(T) ={{n}\choose{2}} - \delta^-_3(T) + 2$, where $\delta^-_3(T) = min\{ d_T^-(x) + d_T^-(y) + d_T^-(w) : \{x, y, w\}\subseteq V(T)\}$.  Moreover, if the arcs of $T$ are colored with $h(T) - 1$ colors, and there is no heterochromatic out-directed spanning tree of $T$, then there is a triple $\{x, y, w\} \subseteq V(T)$ such that  $\delta^-_3(T) = d_T^-(x) + d_T^-(y) + d_T^-(w)$, all the in-arcs of $x, y,$ and $w$ receive the same color and each of the remaining arcs of $T$ receives a new different color.
\end{theorem}

\section{Notation and Preliminary Results}

Let $D = (V(D), A(D))$ be a digraph and $x$ be a vertex of $D$. We denote by $N_D^+(x)=\{v \in V(D): xv \in A(D) \}$ and $N_D^-(x)=\{v \in V(D): vx \in A(D) \}$  the sets of  {\it out-neighbors}  and of  {\it in-neighbors} of $x$ in $D$, respectively.  Likewise, we denote by  $d_D^+(x)=|N_D^+(x)|$ and $d_D^-(x)=|N_D^-(x)|$  the {\it ex-degree} and the {\it in-degree} of $x$ in $D$, respectively. 

For every $Q\subseteq V(D)$, let $F_D^+(Q)= \{zw\in A(D) : z\in Q\hbox{ and } w\in V(D)\setminus Q\}$, $F_D^-(Q) = \{wz : z\in Q \hbox{ and } w\in V(D)\setminus Q\}$ and $F_D(Q) = F_D^+(Q)\cup F_D^-(Q)$. Given $x\in V(D)$ the sets $F_D^+(\{x\}), F_D^-(\{x\})$ and $F_D(\{x\})$  are called the set of ex-arcs, the set of in-arcs and the set of arcs of $x$, respectively. For $Q, R \subseteq V(D)$, we denote by $(Q\rightarrow R)$ the set $ \{ xy \in A(D) : x\in Q \hbox{ and } y \in R\}$. 
 
Let $\Gamma : A(D) \rightarrow C$ be an arc-coloring of $D$. We denote by  $C(x)$ the set of colors that appear only on arcs of $D$ incident to  $x$, and by $c(x)$ the number of colors in $C(x)$. A color $i \in C$  is a $\Gamma_D$-{\it singular color}  if $|\Gamma^{-1}(i)| = 1$.
 
For any vertex $x \in V(D)$ and any arc $wy \in A(D)$, we denote by $D-x$ and $D-wy$ the digraphs obtained from $D$ by deleting the vertex $x$ and the arc $wy$, respectively. For an arc $zy \notin A(D)$, $D+zy$ is the digraph obtained from $D$ by adding the arc $zy$.
 
We say that a vertex $z\in V(D)$ is reachable from a vertex $x$ in $D$  if there is a directed path in $D$ from $x$ to $z$. 

Let  $\delta^-_3(D) = min\{ d_D^-(x) + d_D^-(y) + d_D^-(w) : \{x, y, w\}\subseteq V(D)\}$.

\begin{lemma}\label{cotainf}  
Let $T$ be a tournament of order $n\geq 3$. Then $$h(T) \geq {{n}\choose{2}} - \delta^-_3(T) + 2.$$ 
\end{lemma}
\begin{proof} 
Let $\{x, y, w\}\subseteq V(T)$ such that $d_T^-(x) + d_T^-(y) + d_T^-(w) = \delta^-_3 (T)$  and color the arcs of $T$ with ${{n}\choose{2}} - \delta^-_3(T) + 1$ colors in the following way: all the in-arcs of $x, y$  and $w$ receive the same color, say color black, and the remaining  ${{n}\choose{2}} - \delta^-_3(T)$ arcs receive  ${{n}\choose{2}} - \delta^-_3(T)$ new colors.  

Given an out-directed spanning tree $S$ of $T$ we can assume,  without loss of generality,  that neither  $x$ nor $y$ is the root of $S$, and therefore $d^-_S(x) =  d^-_S(y)= 1$. From here we see that $S$ has at least two black arcs, thus $S$ is not heterochromatic and the lemma follows.
\end{proof}

\section{Proof of Theorem \ref{teorema}}

Lemma \ref{cotainf} gives the lower bound for $h(T)$ in Theorem  \ref{teorema}. The proof of the upper bound and of the remainder of the theorem is by induction on $n$. For better readability, we break down the proof into several lemmas. 

It is not hard to see that if $T$ is a tournament of order 3, and $\Gamma$ is and arc-coloring of $T$ with no heterochromatic out-directed spanning tree,  then $\Gamma$ uses $1 = {{3}\choose{2} }- \delta^-_3(T) +1$ color.  It is also clear that $V(T) = \{ x,y,z \} $ is such that $d_T^-(x) + d_T^-(y) + d_T^-(z) = \delta^-_3 (T) = 3$   and that the three  in-arcs of $x$, $y$ and $z$  receive the same color. This shows that Theorem~\ref{teorema}  holds for tournaments of order 3.

Let $T$ be a tournament of order  $n \geq 4$. For the rest of the proof we assume as inductive hypothesis  that Theorem~\ref{teorema} holds for every tournament of order $m$, with $3\leq m < n$.   

Let $\Gamma$ be an arc-coloring of $T$ which uses $h(T)-1$ colors and produces no heterochromatic out-directed spanning trees of $T$.  Observe that by Lemma~\ref{cotainf}, $h(T) \geq {{n}\choose{2}} - \delta^-_3(T) + 2$ and therefore the number of colors in $\Gamma[A(T)]$ (from now on $\Gamma[T]$ for short)  is at least  $ {{n}\choose{2}} - \delta^-_3(T) + 1.$  

A vertex $x$ of $T$ is of {\it type} 1  if there is an in-arc $e$ of $x$ such that $\Gamma(e) \in C(x)$; of {\it type}  2 if  none of the in-arcs of $x$ receive a color in $C(x)$ and there are at least two in-arcs of $x$ which receive different colors; and of  {\it type} 3  if none of the in-arcs of $x$ receive a color in $C(x)$ and all the in-arcs of $x$ receive the same color.

The next  three  lemmas will show some  properties of the vertices of type 1 and 2,  and that there are at most $n-2$ vertices of type 1. With these at hand,  we will return to the proof of Theorem \ref{teorema}. 

\begin{lemma}\label{c(x)1}
If $x$ is a vertex of $T$ of type $1$, then  $c(x) \geq n-4$. 
\end{lemma}
\begin{proof} 
Since $x$ is of type 1, there is an arc $yx \in A(T)$ such that $\Gamma(yx)\in C(x)$. Since  $\Gamma(yx) \not\in \Gamma[T - x ]$, the tournament $T-x$ has no  heterochromatic out-directed spanning tree $S$, otherwise $S+yx$ would be a heterochromatic out-directed spanning tree of $T$,   which is not possible. Therefore, by our  induction hypothesis, the number of colors appearing in   $\Gamma[T-x]$ is at most  ${{n-1}\choose{2}} - \delta^-_3(T-x) + 1$. Thus $$c(x) \geq {{n}\choose{2}} - \delta^-_3(T) + 1 - \Bigg( {{n-1}\choose{2}} - \delta^-_3(T-x) + 1 \Bigg) = n - 1 + \delta^-_3(T-x) - \delta^-_3(T) .$$ Now just observe that $ \delta^-_3(T) - \delta^-_3(T-x)\leq 3$ and therefore $c(x)\geq n-4$. 
\end{proof}

\begin{lemma}\label{c(x)2}
If $x$ is a vertex of $T$ of type $2$, then $d_T^+(x) \geq c(x) = n-4$.
\end{lemma}
\begin{proof} 
By definition of type 2, none of the colors of the in-arcs of $x$ is in $C(x)$,  so all the colors from $C(x)$ appear on the out-arcs of $x$ and therefore  $d_T^+(x) \geq c(x)$.  Also by definition, there are vertices $y_1, y_2 \in N_T^-(x)$ such that $c_1= \Gamma(y_1x) \not = \Gamma(y_2x) =c_2$ with $c_1, c_2\not\in C(x)$.   

Let $\Gamma'$ be an arc-coloring of $T-x$  obtained from $\Gamma$ by recoloring the arcs of color $c_2$ with  color $c_1$. 
 
Suppose $T-x$ has an out-directed spanning tree $S$ which is heterochromatic with respect to  $\Gamma'$. Clearly  $S$ is also  heterochromatic with respect to  $\Gamma$ and it is such that either color $c_1$ or color  $c_2$  does not appear in $\Gamma[S]$. Thus, either $S+y_1x$ or $S+y_2x$ is a heterochromatic out-directed spanning tree of $T$ with respect to $\Gamma$, which is not possible. Therefore $T-x$ has no heterochromatic out-directed spanning tree with respect to $\Gamma'$. By our induction hypothesis, there are at most ${{n-1}\choose{2}} - \delta^-_3(T-x) + 1$ colors in $\Gamma'[T-x]$. It follows at most ${{n-1}\choose{2}} - \delta^-_3(T - x ) + 2$ colors of $\Gamma $ are used in $T-x$ which implies $$c(x) \geq {{n}\choose{2}} - \delta^-_3 (T) + 1 - \Bigg ( {{n-1}\choose{2}} - \delta^-_3(T-x) + 2\Bigg) = n-2 - \delta^-_3(T) + \delta^-_3(T-x)\geq n-5.$$
  
If $c(x) = n-5$, each of the following must happen: {\it i}) $ \delta^-_3(T) - \delta^-_3(T - x )\ = 3$;  {\it ii})   $|\Gamma [T - x ]| = {{n-1}\choose{2}} - \delta^-_3(T - x ) + 2$;  {\it iii})  $|\Gamma' [T - x ]| = {{n-1}\choose{2}} - \delta^-_3(T - x ) + 1$ and {\it iv}) $T-x$ has no heterochromatic out-directed spanning tree with respect to $\Gamma'$. 
  
By  induction $h(T - x)={{n-1}\choose{2}} - \delta^-_3(T-x) + 2$ and therefore, according to {\it iii}), $\Gamma'$ is an arc-coloring of $T-x$ with $h(T-x)-1$ colors. Also by induction, there is a triple $\{x_1, x_2, x_3\}  \subseteq V(T - x )$ such that $\delta^-_3(T - x ) = d_{T - x }^-(x_1) + d_{T - x }^-(x_2) + d_{T - x }^-(x_3)$, all the in-arcs of $x_1, x_2,$ and $x_3$ have the same color in $\Gamma' $ and each of the remaining arcs of $T - x $ has a singular color in $\Gamma'$. 

Recall that there are arcs in $T-x$ with colors $c_1$ and $c_2$, since  $c_1, c_2\not\in C(x)$. Therefore $c_1$ is the non-singular color in $\Gamma' $ and all the in-arcs of  $x_1$,  $x_2,$ and $x_3$ have  color $c_1$ in $\Gamma'$. This implies that  all the in-arcs  of  $x_1, x_2,$ and  $x_3$ have color $c_1$ or color $c_2$ in $\Gamma$;   and each of the remaining arcs of $T - x $ has a singular color in  $\Gamma$. 

By {\it i}), $ \delta^-_3(T) - \delta^-_3(T - x )\ = 3$ and this implies  $\{x_1, x_2, x_3\}\subseteq N_T^+(x)$. Therefore $\{y_1, y_2\} \subseteq  V(T)\setminus \{x, x_1, x_2, x_3\}$ and  $N^+_T(x) \subseteq V(T)\setminus \{x, y_1, y_2\}$.   Since $c(x) = n-5$, it follows that  there is at least one vertex $z\in \{x_1, x_2, x_3\}$ such that $\Gamma(xz)\in C(x)$. Without loss of generality assume $z=x_1$. 

\noindent {\it Case 1. } $\{ \Gamma(xx_1), \Gamma(xx_2), \Gamma(xx_3)\} \cap  C(x) = \Gamma(xx_1)$.
  
\noindent   The ex-arcs of $x$ with the other $(n-6)$ colors of $C(x)$ appear in $(x\rightarrow [V(T)\setminus\{x, x_1, x_2, x_3\}])$. Thus $N^-_T(x) = \{y_1, y_2\}$ and $\delta^-_T (x) = 2$.  Since  $ \delta^-_3(T) - \delta^-_3(T - x )\ = 3$, it follows that $\delta^-_3(T) =  d_{T}^-(x_1) + d_{T}^-(x_2) + d_{T}^-(x_3)$ and therefore $\delta^-_T(x_i) \leq 2$ for $i=1, 2, 3.$ Since $\{x_1, x_2, x_3\}\subseteq N_T^+(x)$, it follows that $\{x_1, x_2, x_3\}$ induces a directed cycle with length 3 in $T$ (with colors $c_1$ and $c_2$), and $V(T)\setminus\{x, x_1, x_2, x_3\} \subseteq N^+(x_i)$ for $i=1, 2, 3$, where each of  the arcs in $\bigg(\{x_1, x_2, x_3\}\rightarrow [V(T)\setminus\{x, x_1, x_2, x_3\}]\bigg)$ receives a $\Gamma_{T - x }$-singular color (none of them a color in $C(x)$).   Therefore,  the tournament $H$ induced by $V(T)\setminus \{x, x_1\}$ is a heterochromatic tournament in which either $c_1$ or $c_2$ appear, but not both.  Thus, in $H$ there is a hamiltonian heterochromatic path $P$ where, without loss of generality,  color $c_1$ does not appear. Therefore $E(P) \cup \{ y_2x \} \cup \{ xx_1 \}$ induces a heterochromatic out-directed spanning tree  of $T$ which  is not possible. 
  
\noindent {\it Case 2. } $|\{ \Gamma(xx_1), \Gamma(xx_2), \Gamma(xx_3)\} \cap  C(x)| \geq 2$.

\noindent  Suppose $\Gamma(xx_2)\in C(x)$ and  $\Gamma(xx_1) \not = \Gamma(xx_2)$. Consider the tournament $H$ induced by $V(T)\setminus\{x, x_1, x_2\}$ and let $P$ be a hamiltonian path in $H$. Except for the in-arcs of $x_3$, which receive color $c_1$ or $c_2$,  all the other arcs in $H$ receive $\Gamma_{T - x }$-singular colors. Thus $P$ is a heterochromatic path in which either color $c_1$ or color $c_2$ appear, but not both.  Without loss of generality, suppose color $c_1$ does not appear in $P$.  In this case $E(P) \cup \{y_2x \} \cup \{xx_1, xx_2 \} $ induces a heterochromatic out-directed spanning tree of $T$ which  again is not possible.

From Case 1 and Case 2,  it follows that $c(x)\geq n-4$.  Suppose  $c(x) \geq n-3$. Since $d_T^+(x) \geq c(x)$ and $\Gamma(y_1x), \Gamma(y_2x)\not\in C(x)$,  all the ex-arcs of $x$ receive  different colors and all of them lie in $C(x)$. Since $\Gamma(y_1x)=c_1\not=c_2=\Gamma(y_2x)$ and 
the color of the arc with endpoints $y_1$ and $y_2$ is not in $C(x)$,  it is not hard to see that either  $F^+_T(\{x\})\cup  \{y_1x, y_1y_2\}$ or $F^+_T(\{x\})\cup \{y_1y_2, y_2x\}$ induces a heterochromatic out-directed spanning tree of $T$ which is not possible.  Therefore $c(x) = n-4$ and Lemma \ref{c(x)2} follows. \end{proof}

\begin{lemma} \label{tipo1}  
There are at most $n-2$ vertices of $T$ of type 1.
\end{lemma}
\begin{proof} 
Suppose there are at least $n-1$ vertices of type 1.  Let $D$ be a spanning subdigraph of $T$ with the minimum number of connected components whose arc set is obtained as follows:  choose a set $A$ with $n-1$ vertices of type 1, and  for each vertex $x\in A$,  choose one in-arc of $x$ with a color in $C(x)$.  

Clearly $D$ is heterochromatic. Since there are no heterochromatic out-directed spanning trees of $T$,  $D$ is not connected. Let $D_1, D_2, \dots, D_r$  be the connected components of $D$. Since $D$ has $n$ vertices and $n-1$ arcs and the maximum in-degree of $D$ is 1, it is not hard to see that one connected component, say $D_1$, is an out-directed tree, while, for $i=2,3, \ldots, r$, component $D_i$  contains exactly one directed cycle $C_i$ such that $D-e$ is an outdirected tree for each edge $e$ of $C_i$.  Let $z_1$ be  the root of $D_1$ and notice  that $A= V(T)\setminus \{z_1\}$.

\begin{claim}
 Let $x\in V(C_2)$ ,  $y\in \bigcup\limits_{j\not=2} V(C_j) \cup  \{z_1\}$ and $e$ be the arc with endpoints $\{x, y\}$. If $\Gamma(e)\in C(x)$ then $e$ is an ex-arc of $x$ and   $\Gamma(e)$ is not a $\Gamma_T$-singular color.
\end{claim}

Suppose $\Gamma(e)\in C(x)$.  If $e$ is an in-arc of $x$, the digraph $(D-wx) + e$,  with $wx\in A(C_2)\subseteq A(D)$,  has fewer connected components than $D$ and can be obtained in the same way as $D$ by choosing in $C(x)$ the edge $e$ instead of $wx$, which is a contradiction.  Hence $e$ is an ex-arc of $x$, and therefore an in-arc of $y$.    Let us suppose  $\Gamma(e)$ is a $\Gamma_T$-singular color.   Thus $\Gamma(e)\in C(y)$ and $y$ is of type 1.  On the one hand, if $y\in V(C_j)$ for some $j\not= 2$,  in an analogous way as with the vertex $x$, we reach a contradiction.  On the other hand, if $y=z_1$ the digraph $(D - wx) + e$ (which  has fewer connected components than $D$) can be obtained in the same way as $D$ by choosing the set $A' = (A\setminus\{x\}) \cup \{z\}$ as the set of  $n-1$ vertices of type 1 and choosing the edge $e$ in $C(z_1)$ instead of the edge $wx$ in $C(x)$, which is a contradiction. From here, Claim 1 follows.

Let $x\in V(C_2)$. Since $c(x) =  n-4$ it follows  there are at least $n-7$ arcs incident to $x$ with $\Gamma_T$-singular colors.   Thus, by Claim 1 it follows that $| \{z_1\}\cup  \bigcup\limits_{j\not=2} V(C_j)|\leq 6$ and therefore $r\leq 3$.  Let us suppose $r=2$ and let $e$ be the arc with endpoints $\{ z_1, x\}$.  The color $\Gamma(e)$  must appear in $D$, otherwise $D+e$ is a heterochromatic digraph containing an out-directed spanning tree of $T$ which is a contradiction.  By the choice of the arcs of $D$, $\Gamma(e)\in C(x)$ and there is an arc $wx\in A(C_2)$ with color $\Gamma(e)$, but then $(D-wx) + e$ is a heterochromatic out-directed spanning tree of $T$ which is a contradiction. Thus $r=3$.    Since $c(x) =  n-4$ and $| \{z_1\}\cup V(C_3)|\geq 4$,  there is a color $c\in C(x)$ which only appears in arcs incident to $x$ and with the other endpoint in  $ V(C_3) \cup  \{z_1\}$ . By Claim 1, these arcs are ex-arcs of $x$ and there are at least two of them, since $c$ is not   a $\Gamma_T$-singular color.  Thus there is $y\in V(C_3)$ such that $\Gamma(xy) = c$.  Let $w\in V(C_2)\setminus \{x\}$ and let 
$e$ be the arc with endpoints $\{ z_1, w\}$.   The color $\Gamma(e)$  must appear in $D+xy$, otherwise $D+\{xy, e\}$ is a heterochromatic digraph containing an out-directed spanning tree of $T$ which is a contradiction.  Thus, by the choice of the arcs of $D$ and since $\Gamma(xy)\in C(x)$,  $\Gamma(e)\in C(w)$ and there is an arc $ww'\in A(C_2)$ with color $\Gamma(e)$, but then $(D-ww') + \{xy, e\}$ is a heterochromatic digraph containing an out-directed spanning tree of $T$ which is a contradiction. This ends the proof of Lemma \ref{tipo1}.  \end{proof}

Now we return to the proof of Theorem \ref{teorema}.  First we will show that there is an arc $x_1x_2\in A(T)$ and a vertex $x_3\in V(T)\setminus\{x_1, x_2\}$ such that the spanning subdigraph $D$ of $T$ with set of arcs 
$$A(D) = \Big(A(T)\setminus \bigcup\limits_{i=1}^3 F_T^-(\{x_i\})\Big)\cup \{x_1x_2\}$$ is an heterochromatic spanning subdigraph of $T$  with $h(T)-1$ arcs.  Observe that  these will imply that $$h(T)-1=  |A(D)| = {{n}\choose{2}} - \Big(d^-_T(x_1) + d^-_T(x_2) + d^-_T(x_3)\Big) +1 \leq  {{n}\choose{2}} -\delta_3^-(T) +1$$  which  will prove the first part of the theorem. 

 Recall that if $v$ is a vertex of $T$ of type 3, then all the in-arcs of $v$ recieve the same color. For each such vertex $v$ we denote by $c_v$ the color assigned to every in-arc of $v$. 

Now we will choose a pair of vertices $\{x, y\}$ in the following way:  By Lemma ~\ref{tipo1} there are at least two vertices that are not  of type 1. If there are at least two vertices of type 3, choose $x$ and $y$ to be vertices of type 3 such that $c_x = c_y$ if possible, otherwise chose any two vertices of type 3. If there is exactly one vertex of type 3, choose it together with any vertex of type 2.  Otherwise choose $x$ and $y$ to be vertices of type 2. 

Without loss of generality assume $xy\in A(T)$ and let $c_0 = \Gamma(xy)$. Let  $D$ be a maximal heterochromatic spanning subdigraph of $A(T)\setminus F_T^-[\{x, y\}] \cup xy$ that contains $xy$.  Observe that the number of arcs in $D$ is
 \begin{equation}\label{eq0}|A(D)|=\Gamma[T] - k(x,y),
 \end{equation}
 where $k(x,y)$ is the number of colors that only appear in the set of arcs  $F_T^-[\{x, y\}]$.

\begin{claim}
$k(x, y) = 0$. 
\end{claim}

Suppose $k(x,y)\geq 1$ and let $c_1$  be a color that only appears in the set of arcs $F_T^-(\{x,y\})$.  Since neither $x$ nor $y$ are of type 1, $c_1\not\in C(x)\cup C(y)$, there is a pair of arcs  $\{z_xx, z_yy \}\subseteq F_T^-[\{x, y\}]$   (where $z_x$ and $z_y$ are not necessarily different) such that $\Gamma(z_xx) = \Gamma(z_yy) = c_1$.  Since $y$ is not of type 1 and $\Gamma(xy) = c_0\not = c_1 =\Gamma(z_yy)$,  it follows that $y$ is of type 2. 

Let  $A= \{yx_1, yx_2, \ldots, yx_{c(y)} \}$ be a set of ex-arcs of $y$, all of them with different colors in $C(y)$,   contained in $A(D)$. By Lemma~\ref{c(x)2}, $c(y)= n-4$, since $y$ is of type 2.  Thus $\{xy\}\cup A$ induces a heterochromatic out-directed tree of order $n-2$,  with root $x$ and with colors in $\{c_0\}\cup C(y)$. 

Let $\{w_1, w_2\} = V(T) \setminus\big( \{x, y\} \cup \{x_i : yx_i\in A\}\big)$. Observe that $z_y\in \{w_1, w_2\}$ and, without loss of generality, assume $z_y = w_1$.   Since $y$ is of type 2, by the way $x$ and $y$ were chosen, it follows that neither $w_1$ nor $w_2$ is of type 3. For $i=1, 2$, observe that if $w_i$ is of type 1, then there is an in-arc of $w_i$ with a color in $C(w_i)$, which does not appear in $xy\cup A$. Also notice that if $w_1$ is of type 2, then there are two in-arcs of $w_1$, with different colors such that those colors are not in $C(y)$ and that if $w_2$ is of type 2, then there are two in-arcs of $w_2$, also with different colors, such that at least one of those colors is not in $C(y)$ (maybe $yw_2\in A(T)$ and $\Gamma(yw_2)\in C(y)$). 

In any case, there exist  in-arcs $e_1$ of $w_1= z_y$ and $e_2$ of $w_2$  with different colors,  none of them with color in $C(y)$, none of them with  color $c_1$ (recall that all the arcs of color $c_1$ are in-arcs of $x$ and $y$), and maybe one of them with color $c_0$. Since $\Gamma(z_xx) = c_1$, it follows that $A\cup \{xy, z_xx, e_1, e_2\}$ contains a heterochromatic out-directed spanning tree of $T$ which is not possible and therefore, Claim 2 holds.

Since $k(x,y)=0$ and $|A(D)|=\Gamma[T] - k(x,y)$, we see that the number of arcs in $D$ is $\Gamma[T] = h(T)-1 \geq  {{n}\choose{2}} - \delta^-_3(T) + 1$.  Notice that none of the in-arcs of $x$ are in $A(D)$ and, except for $xy$, none of the in-arcs of $y$ are in $A(D)$.  Let $H\subseteq V(T)$ be the set of vertices which are reachable from $x$ by directed paths in $D$. Since $T$ has no heterochromatic out-directed spanning tree with respect to $\Gamma$, it follows that $W= V(T)\setminus H \not= \emptyset$.  Thus, none of the arcs in $F_T^-(W)$ are present in $D$.  Therefore,   \begin{equation}\label{eq1}|A(D)| = {{n}\choose {2}} - d_T^-(x) - d_T^-(y) - |F_T^-(W)| +1 - \alpha\end{equation}  with $\alpha \geq 0$ (maybe other arcs in $A(T)\setminus \Big( F_T^-(W) \cup F_T^-(\{x, y\})\Big)$ do not appear in $D$).  

Since $$|A(D)| \geq {{n}\choose{2}} - \delta^-_3(T) + 1,$$  it follows from (\ref{eq1})  that  \begin{equation}\label{eq3} \delta^-_3(T) \geq d_T^-(x)  + d_T^-(y) +  |F_T^-(W)|  + \alpha.\end{equation}

It is not hard to see that $|F_T^-(W)| = \sum\limits_{z\in W} d_T^-(z)  - {{|W|}\choose {2}}$ and therefore \begin{equation}\label{eq2}d_T^-(x)  + d_T^-(y) +  |F_T^-(W)|  + \alpha =  \sum\limits_{z\in W\cup\{x,y\}} d_T^-(z)  - {{|W|}\choose {2}}  + \alpha.\end{equation}

On the other hand, by an averaging argument we see that $$\Big(\frac{3}{|W|+2}\Big)  \sum\limits_{z\in W\cup\{x,y\}} d_T^-(z)  \geq \delta^-_3(T) $$ and then, by~(\ref{eq3}) and (\ref{eq2}), 
$$\Big(\frac{3}{|W|+2}\Big)  \sum\limits_{z\in W\cup\{x,y\}} d_T^-(z)  \geq  \sum\limits_{z\in W\cup\{x,y\}} d_T^-(z)  - {{|W|}\choose {2}}  + \alpha$$

Therefore  $${{|W|}\choose {2}}\geq \Big(\frac{|W|-1}{|W|+2}\Big)  \sum\limits_{z\in W\cup\{x,y\}} d_T^-(z)  + \alpha, $$ but since $\sum\limits_{z\in W\cup\{x,y\}} d_T^-(z) \geq  {{|W|+2}\choose {2}}$, we see that $${{|W|}\choose {2}}\geq \frac{(|W|+1)(|W|-1)}{2}+\alpha$$ and hence \begin{equation}\label{eq4} 0 \geq \frac{|W| -1}{2} +\alpha.\end{equation}

\noindent Since $W\not=\emptyset$,  $\frac{|W|-1}{2}\geq 0$ and then,  from~(\ref{eq4}) it follows that  $|W| = 1$ and $\alpha = 0$.   Let $\{w\} = W$. Clearly $|F_T^-(W)| = d_T^-(w)$, and by~(\ref{eq3})  we see that $$\delta^-_3(T) \geq  d_T^-(x)  + d_T^-(y) +  |F_T^-(W)|  + \alpha = d_T^-(x)  + d_T^-(y) +  d_T^-(w)$$ which,  by definition of $ \delta^-_3(T)$ implies that  \begin{equation}\label{eq8}\delta^-_3(T) =  d_T^-(x)  + d_T^-(y) +  d_T^-(w).\end{equation}  

Since $\alpha =0$,  it follows that  all the arcs of $T$ are present in $D$  except for the in-arcs of $x$, the in-arcs of $w$ and,  besides the arc $xy$, all the in-arcs of $y$.   Thus  $$A(D) = \Big(A(T)\setminus \bigcup\limits_{z\in\{x,y,w\}} F_T^-(\{z\})\Big)\cup \{xy\}$$ and 
$$|A(D)| = {{n}\choose{2}} - (  d^-(x)  + d^-(y) +  d^-(w))+ 1 =  {{n}\choose{2}} - \delta^-_3(T) + 1,$$ and since $|A(D)|= h(T) -1$ it follows that \begin{equation}\label{eq5}  h(T) =  {{n}\choose{2}} - \delta^-_3(T) + 2.\end{equation}

From  here,  to end the proof  of Theorem \ref{teorema}  just remain to show that all the in-arcs of $x, y,$ and $w$ receive the same color.  For this, first we will prove that all the in-arcs of $w$ receive color $c_0$.  Let suppose there is an arc $zw\in A(T)$ such that $\Gamma(zw)= c_3\not =  c_0$.  Since all the colors in $\Gamma[T]$ are present in $A(D)$,   there is an arc $z'w' \in A(D)$ such that $\Gamma(z'w') = c_3$. Notice that  $w'\not\in\{x, y, w\}$, since no in-arcs of $x$ nor $w$ are present in $D$ and  the only in-arc of $y$ in $D$ has color $c_0$.  Let $D' = \Big(D\setminus z'w'\Big) \cup zw$.  Observe that both vertices $z$ and $z'$ are reachable from $x$ in both digraphs $D$ and $D'$.  Also notice that $D'$ is a maximal heterochromatic spanning subdigraph of $A(T)\setminus F_T^-[\{x, y\}] \cup xy$ that contains $xy$.  Thus, by an analogous procedure as for $D$, we find that in $D'$ there is a vertex $v$ such that all the arcs of $T$ are present  in $D'$ with exception of the in-arcs of $x$, the in-arcs of $v$ and,  besides the arc $xy$, all the in-arcs of $y$.  Since $w'$ has an in-arc missing in $D'$ and $v\not\in\{x,y\}$, it follows that $v = w'$. 

Since $w\not = w'$,  either $ww'\in A(T)$  or $w'w\in A(T)$.   If $w'w\in A(T)$, $w'w\not\in A(D)$  but $w'w\in A(D')$, and since $D' = \Big(D\setminus z'w'\Big) \cup zw$ it follows that   $zw=w'w$ and $w'=z$ which is not possible since $z$ is reachable from $x$ in $D'$ and $w'$ is not reachable from $x$ in $D'$. In an analogous way, if $ww'\in A(T)$, $ww'\not\in A(D')$ but $ww'\in A(D)$ and  then $z'w'=ww'$ and $w=z'$, which is not possible since $z'$ is reachable from $x$ in $D$ and $w$ is not. 

Therefore all the in-arcs of $w$ receive color $c_0$. Thus $w$ is a vertex of type 3, and,  by the way the pair $\{x, y\}$ were chosen,  this implies that  $\{w,x, y\}$ is a triple of vertices of type 3, and since $c_0=c_x = c_w$,  again,  by the way the pair $\{x, y\}$ were chosen, $c_y=c_x$. Therefore  all the in-arcs of the triple $\{x,y,w\}$ receive the same color $c_0$ and this ends the proof of Theorem \ref{teorema}.

\end{document}